\documentclass[a4paper,12pt]{amsart}
\usepackage{amssymb}
\usepackage{amsmath, amsthm, amscd, amsfonts, amssymb, graphicx, color}
\usepackage{amsmath, amsthm, amscd, amsfonts, amssymb, graphicx, color}

\newtheorem{theorem}{Theorem}[section]

\newtheorem{corollary}[theorem]{Corollary}
\theoremstyle{definition}
\newtheorem{definition}[theorem]{Definition}
\newtheorem{example}[theorem]{Example}

\theoremstyle{remark}
\newtheorem{remark}[theorem]{Remark}
\numberwithin{equation}{section}

\title[Linear dynamics of cosine operator functions on Solid Spaces]
{Linear dynamics of cosine operator functions on solid Banach function spaces}
\author[S.M. Tabatabaie]{Seyyed Mohammad Tabatabaie}
\address{Department of Mathematics, University of Qom, Qom, Iran}
\email{sm.tabatabaie@qom.ac.ir}

\author[S. Ivkovi\'{c}]{Stefan Ivkovi\'{c}}
\address{The Mathematical Institute of the Serbian Academy of Sciences and Arts,
	p.p. 367, Kneza Mihaila 36, 11000 Beograd, Serbia}
\email{stefan.iv10@outlook.com }



\date{\today}

\begin{document}

\maketitle

\begin{abstract}
In this paper, we give some sufficient and necessary conditions for cosine operator functions on solid Banach function spaces to be chaotic or topologically transitive.
\end{abstract}

\baselineskip17pt

\section{Introduction}
Linear dynamics of operators have been studied in many articles during several decades; see \cite{gpbook} and \cite{bmbook} as monographs on this topic.
Among other concepts, hypercyclicity, topological transitivity and topological mixing, as important linear dynamical properties of bounded linear operators, have been investigated in many research works; see \cite{an97, bg99, kuchen17, shaw, kostic} and their references.
Specially, hypercyclic weighted shifts on $\ell^p(\Bbb{Z})$ were characterized in \cite{sa95,ge00}, and then  C-C. Chen and 
C-H. Chu, using aperiodic elements of locally compact groups, extended the results in \cite{sa95} to weighted translations on Lebesgue spaces in the context of a second countable group \cite{cc11}. Afterwards, many other linear dynamics in connection with this theme have been studied; see \cite{chen11, chen141, chta1, chta2, chta3}. Recently, the study of hypercyclicity  of operators on Orlicz spaces, as an extension of Lebesgue spaces, has been done; see \cite{aa17, ccot, cd18}.
 In \cite{bmi08} a sufficient condition was given for a cosine operator function  generated by a strongly continuous translation group on some weighted Lebesgue space $L^p(\mathbb{R})$ to be topologically transitive. Also, in \cite{kalmes10} Kalmes characterized transitive and mixing cosine operator functions generated by second order partial differential operators on $L^p(\Omega,\mu)$, where $\Omega\subset \mathbb{R}^n$ and $\mu$ is a Borel measure on $\Omega$. 
Then,  in \cite{chen141,chen15} C-C. Chen studied transitive cosine operator functions generated by a weighted translation operator on $L^p(G)$, where $G$ is a locally compact group, and gave a characterization of mixing cosine operator functions; see also \cite{chache}. In \cite{ctm3} we consider actions of locally compact groups on measure spaces and give some sufficient conditions for cosine operator functions generated by group actions, to be chaotic.
We recall that a cosine operator function on a Banach space $\mathcal X$ is a mapping $\mathcal C$ from $\mathbb R$ into the space of all bounded linear  operators on $\mathcal{X}$ such that $\mathcal{C}(0) =I$ and  $2\mathcal{C}(t)\mathcal{C}(s) =\mathcal{C}(t+s)+\mathcal{C}(t-s)$ for all $s,t\in\mathbb{R}$. See \cite{but,sha} for the roles of this class of mappings for investigation of semigroups.
Recently, we have studied dynamics of weighted translation operators on general solid Banach function spaces, and  have given a characterization for weighted translation operators to be disjoint hypercyclic and topologically mixing on solid Banach function spaces. In this paper, we continue this study by investigating some sufficient and necessary conditions for cosine operator functions on such general Banach function spaces to be chaotic or topologically transitive. The obtained results in addition to the previous modes, will cover new cases such as operators on tiled subspaces of Morrey spaces; see Example \ref{ex}. 
 For convenience of readers, in section 2 we recall some definitions, notations and examples, and also introduce the new setting.  
\section{Preliminaries}
If $\mathcal X$ is a Banach space, the set of all bounded linear operators from $\mathcal X$ into $\mathcal X$ is denoted by $B(\mathcal X)$. Also, we denote $\Bbb{N}_0:=\Bbb{N}\cup\{0\}$.
\begin{definition}
	Let $\mathcal X$ be a Banach space. A sequence $(T_n)_{n\in \Bbb{N}_0}$  of operators in $B(\mathcal X)$ is called {\it topologically transitive} if for each non-empty open subsets $U,V$ of
	${\mathcal X}$, $T_n(U)\cap V\neq \varnothing$ for some $n\in \Bbb{N}$. If $T_n(U)\cap V\neq \varnothing$ holds from some $n$ onwards, then
	$(T_n)_{n\in \Bbb{N}_0}$ is called {\it topologically mixing}. 
\end{definition}
\begin{definition}
	Let $\mathcal X$ be a Banach space. A sequence $(T_n)_{n\in \Bbb{N}_0}$  of operators in $B(\mathcal X)$ is called {\it hypercyclic} if there is an element $x\in\mathcal X$ (called \emph{hypercyclic vector}) such that the orbit $\mathcal O_x:=\{T_nx:\,n\in\mathbb N_0\}$ is dense in $\mathcal X$. The set of all hypercyclic vectors of a sequence $(T_n)_{n\in \Bbb{N}_0}$ is denoted by $HC((T_n)_{n\in \Bbb{N}_0})$. If $HC((T_n)_{n\in \Bbb{N}_0})$ is dense in $\mathcal X$, the sequence $(T_n)_{n\in \Bbb{N}_0}$ is called \emph{densely hypercyclic}. An operator $T\in B(\mathcal X)$ is called \emph{hypercyclic} if the sequence $(T^n)_{n\in \mathbb N_0}$ is hypercyclic.
\end{definition}
Note that a sequence $(T_n)_{n\in \Bbb{N}_0}$  of operators in $B(\mathcal X)$ is topologically transitive if and only if it is densely hypercyclic \cite{gro}. Also, a Banach space admits a hypercyclic operator if and only if it is separable and infinite-dimensional \cite{an97,bg99}. So, in this paper we assume that Banach spaces are separable and infinite-dimensional.
\begin{definition}
	Let $\mathcal X$ be a Banach space, and $(T_n)_{n\in \Bbb{N}_0}$  be a sequence of operators in $B(\mathcal X)$. A vector $x\in \mathcal X$ is called a {\it periodic element} of $(T_n)_{n\in \Bbb{N}_0}$ if there exists a constant $N\in\mathbb N$ such that for each $k\in\mathbb N$, $T_{kN}x=x$. The set of all periodic elements of $(T_n)_{n\in \Bbb{N}_0}$ is denoted by
${\mathcal P}((T_n)_{n\in \Bbb{N}_0})$. The sequence $(T_n)_{n\in \Bbb{N}_0}$ is called {\it chaotic} if $(T_n)_{n\in \Bbb{N}_0}$ is topologically transitive and ${\mathcal P}((T_n)_{n\in \Bbb{N}_0})$ is dense in ${\mathcal X}$.
\end{definition}

In sequel, the set of all Borel measurable complex-valued functions on a topological space $X$ and the set of all compactly supported continuous complex-valued functions on $X$ are denoted by $\mathcal M_0(X)$ and $C_c(X)$, respectively. Also, the support of each function $f:X\rightarrow\mathbb C$ is denoted by ${\rm supp}(f)$, and $\chi_A$ denotes the characteristic function of a set $A$.
\begin{definition}
	Let $X$ be a topological space and $\mathcal{F}$ be a linear subspace of $\mathcal M_0(X)$. If $\mathcal F$ equipped with a given norm $\|\cdot\|_{\mathcal F}$ is a Banach space, we say that $\mathcal F$ is a \emph{Banach function space on $X$}.
\end{definition}
\begin{definition}
	Let $\mathcal{F}$ be a Banach function space on a topological space $X$, and  $\alpha:X\longrightarrow X $ be a Borel measurable bijection which its inverse $\alpha^{-1}$ is also Borel measurable. We say that $\mathcal{F}$ is \emph{$\alpha$-invariant} if for each 	$f\in \mathcal{F} $ we have 
	$f\circ \alpha^{\pm 1} \in \mathcal{F}$ and 	$\|f\circ \alpha^{\pm 1} \|_{\mathcal{F}}=\|f\|_{\mathcal{F}}$.
\end{definition}
\begin{definition}
	A Banach function space $\mathcal{F}$ on $X$ is called \emph{solid} if for each $f\in \mathcal{F}$ and $g\in\mathcal M_0(X)$, satisfying $|g|\leq |f|$, we have $g \in \mathcal{F}$
	and $\|g\|_{\mathcal{F}}\leq \|f\|_{\mathcal{F}}$.
\end{definition}
\begin{remark}\label{rem1}
	Suppose that a locally compact group $G$ continuously acts on a topological space $X$ by the mapping $(a,x)\mapsto ax$ from $G\times X$ into $X$, and $\mu$ is a Radon measure on $X$ which is invariant related to this action. Fix an element $a\in G$, and define 
	\begin{equation}
	\alpha_a:X\rightarrow X,\quad \alpha_a(x):=ax
	\end{equation}
	for all $x\in X$. 
	Then, for each Young function $\Phi$,  the Orlicz space $(L^\Phi(X,\mu),\|\cdot\|_\Phi)$ is an $\alpha_a$-invariant solid Banach function space on $X$, where $\|\cdot\|_\Phi$ is the Orlicz norm. Also, if $\Phi$ is $\Delta_2$-regular, then $C_c(X)$ is dense in $L^\Phi(X,\mu)$. For more details about Orlicz spaces we refer to monographs \cite{Rao, rao}.
\end{remark}
Now, we are ready to define weighted translation operators on a solid space, which are generalizations of the usual weighted translations on Lebesgue or Orlicz spaces in the context of locally compact groups.
\begin{definition}
	Let $X$ be a topological space and 	$\alpha:X\longrightarrow X$ be a Borel measurable bijection which its inverse is Borel measurable. Let $\mathcal{F}$ be an $\alpha$-invariant solid Banach function space on $X$. Let 
	$w:X\longrightarrow (0,\infty)$ be a bounded Borel measurable function (called \emph{weight}). Then, the operator $T_{\alpha , w}$  is defined on $\mathcal F$ by
	\begin{equation*}
		T_{\alpha , w}: \mathcal{F} \longrightarrow  \mathcal{F}, \quad
		T_{\alpha , w} f:= w \cdot (f\circ\alpha)
	\end{equation*}
	for all  $f\in \mathcal{F}$, and called a \emph{(generalized) weighted translation operator}.
\end{definition}
Easily, one can see that by the above assumptions $T_{\alpha , w}$ is well-defined and $\|T_{\alpha,w}\|\leq \|w\|_{\sup}$.  
If $\frac{1}{w}$ is also bounded, then $T_{\alpha , w}$ is invertible and we have 
\begin{equation*}
	T_{\alpha , w}^{-1}f=\frac{f\circ \alpha^{-1}}{w\circ \alpha^{-1}},\quad(f\in\mathcal F).
\end{equation*}
Simply we denote $S_{\alpha , w}:=T_{\alpha , w}^{-1}$. Also, for each $n\in\mathbb{Z}$, we define
$$C^{(n)}_{\alpha,w}:=\dfrac{1}{2}(T^n_{\alpha,w}+S^n_{\alpha,w}).$$

Note that for each $n\in\mathbb{Z}$, $C^{(-n)}_{\alpha,w}=C^{(n)}_{\alpha,w}$. Setting $\mathcal{C}(n):=C^{(n)}_{\alpha,w}$ for all $n\in\mathbb{Z}$, one can consider this sequence as a cosine operator function  generated by above weighted translation operators on $\mathcal{F}$. In this paper, we study linear dynamics of the sequence $(C^{(n)}_{\alpha,w})_{n\in \mathbb{N}_0}$.
\begin{remark}
If $w$ and $\frac{1}{w}$ are weights, the inverse of a weighted translation operator  $T_{\alpha,w}$ is also a weighted translation operator. In fact, $ S_{\alpha , w}=T_{\alpha^{-1}, \frac{1}{w \circ \alpha^{-1}}}$. Moreover, if $T_{\alpha_{1}, w_{1}} $ and $T_{\alpha_{2}, w_{2}}  $ are two weighted translation operators, then 
$$T_{\alpha_{2}, w_{2}}\circ T_{\alpha_{1}, w_{1}} = T_{\alpha_{1}  \circ \alpha_{2} , w_{2}(w_{1} \circ \alpha_{2})},$$ so the composition of two weighted translation operators is again a weighted translation operator.
	By some calculation one can see that for each $n\in \mathbb{N}$ and $f\in \mathcal{F}$,
	$$T^{n}_{\alpha,w}f =\left(\prod_{j=0}^{n-1}(w\circ \alpha^{j})\right) \cdot (f\circ \alpha^{n})$$
	and 
	$$S^{n}_{\alpha,w}f =\left(\prod_{j=1}^{n}(w\circ \alpha^{-j})\right)^{-1}\cdot (f\circ \alpha^{-n}).$$
	 \end{remark}
 The following concept plays a key role in the proofs.
\begin{definition}
	Let $X$ be a topological space. Let $\alpha:X\longrightarrow X$ be invertible, and $\alpha,\alpha^{-1}$ be Borel measurable. We say that $\alpha$ is \emph{aperiodic} if for each compact subset $K$ of $X$, there exists a constant $N>0$ such that for each $n\geq N$, we have $K \cap \alpha^{n}(K)=\varnothing$, where $\alpha^{n}$ means the $n$-fold combination of $\alpha$. 
\end{definition}
\begin{example}\label{ex1}
	An element $a$ in a locally compact group $G$ is called {\it aperiodic} (or \emph{non-compact} in \cite{rossbook}) if the closed subgroup of $G$ generated by $a$ is not compact.   If  $G$ is second countable, then $a\in G$ is aperiodic if and only if $\alpha_a$ is an aperiodic function, where $\alpha_a$ is a function on $G$ defined  by $\alpha_a(x):=ax$ for all $x\in G$; see \cite[Lemma 2.1]{cc11}.  
	As mentioned in \cite{cc11}, in many locally compact groups such as the additive group $\Bbb R^n$, the Heisenberg group and the affine
	group, all elements except the identity are aperiodic.
\end{example}

For convenience, we collect some  required conditions on the Banach function space $\mathcal{F}$ below.

\begin{definition}\label{condition}
	Let $X$ be a topological space, $\mathcal F$ be a Banach function space on $X$, and $\alpha$ be a bijection from $X$ onto $X$ such that $\alpha$ and $\alpha^{-1}$ are Borel measurable. We say that $\mathcal F$ satisfies condition $\Omega_\alpha$ if:
	~\	\begin{enumerate}
		\item $\mathcal{F}$ is solid and $\alpha$-invariant;
		\item for each compact set	$E\subseteq X$ we have $\chi_{E}\in \mathcal{F}$;
		\item $\mathcal{F}_{bc}$ is dense in $\mathcal F$, where $\mathcal{F}_{bc}$ is the set of all bounded compactly supported functions in $\mathcal{F}$.
	\end{enumerate}
\end{definition}

\begin{example}
	Consider the notations and assumptions given in Remark \ref{rem1}. Let also $\Phi$ be a $\Delta_2$-regular Young function and $X$ be separable.  Then, the  Orlicz space $L^\Phi(X,\mu)$ satisfies the condition $\Omega_{\alpha_a}$; see \cite[Theorem 1 page 87]{Rao}.
\end{example}
\begin{example}\label{ex}
	Let $1\le q \le p < \infty$.
	The Morrey norm of each $L^{q}_{\rm loc}({\mathbb R}^n)$-function $f$ is defined by
	\begin{equation}\label{eq:130709-1A}
	\| f \|_{{\mathcal M}^p_q}
	:=
	\sup\limits_{(x,r) \in \mathbb R^n\times(0,\infty)}
	|B(x,r)|^{\frac{1}{p}-\frac{1}{q}}
	\left(
	\int_{B(x,r)}|f(y)|^{q}{\rm d}y
	\right)^{\frac{1}{q}},
	\end{equation}
	where  $B(x,r)$ is the open ball centered in $x$ with radius $r$. Then, the set of all $L^{q}_{\rm loc}({\mathbb R}^n)$-functions $f$ with $\| f \|_{{\mathcal M}^p_q}<\infty$ is denoted by ${\mathcal M}^p_q(\mathbb{R}^n)$ and is called a \emph{Morrey space}. Morrey spaces are generalization of usual Lebesgue spaces. In fact, for each $1\le p < \infty$ we have $\mathcal M_{p}^p(\mathbb R^n)=L^p(\mathbb R^n)$. For more details and references see \cite{Triebel14, Sawano17, SST11}. For each  $1\le q < p < \infty$, the Morrey space ${\mathcal M}^p_q(\mathbb{R}^n)$ is not separable.
	However, if we denote the ${\mathcal M}^p_q({\mathbb R}^n)$-closure 
	of $L^{\infty}_{\rm c}({\mathbb R}^n) \cap {\mathcal M}^p_q({\mathbb R}^n)$ by $\widetilde{\mathcal M}^p_q({\mathbb R}^n)$,
	where $L^{\infty}_{\rm c}({\mathbb R}^n)$ is
	the set of all functions in $L^{\infty}({\mathbb R}^n)$ with compact support, then the Banach space $\widetilde{\mathcal M}^p_q({\mathbb R}^n)$ is separable and infinite-dimensional. Also, $\left(\widetilde{\mathcal M}^p_q({\mathbb R}^n)-L^p({\mathbb R}^n)\right)\cup\{0\}$ contains an infinite-dimensional closed subspace of $\widetilde{\mathcal M}^p_q({\mathbb R}^n)$; see \cite{sata} for more details about this space. As in Example \ref{ex1}, for each non-zero element $a\in\mathbb{R}^n$, the function $\alpha_a$ defined by $\alpha_a(x):=x+a$ for all $x\in\mathbb{R}^n$
	is an aperiodic function, and the Banach space $(\widetilde{\mathcal M}^p_q({\mathbb R}^n),\|\cdot\|_{{\mathcal M}^p_q})$ satisfies the condition $\Omega_{\alpha_a}$.
\end{example}
\section{main results}
The below result helps us to give a necessary condition for density of the set of all periodic elements of $(C_{\alpha,w}^{(n)})_{n\in \mathbb{N}_0}$.
\begin{theorem}\label{main}
	Let $\alpha$ be an aperiodic function on a topological space $X$, and $\mathcal F$ be a Banach function space on $X$ satisfying the condition $\Omega_{\alpha}$. Let $ w,\frac{1}{w}$ be weight functions on $X$. Then ${\rm (i)}\Rightarrow{\rm (ii)}$.
	\begin{enumerate}
		\item [{\rm(i)}] The set $\mathcal{P}((C_{\alpha,w}^{(n)})_{n\in \mathbb{N}_0})$ is dense in $\mathcal{F}$, and
		for each $f\in\mathcal{P}((C^{(n)}_{\alpha,w})_{n\in \mathbb{N}_0})$, we have $\lim_{n\rightarrow\infty}S^{n}_{\alpha,w}f=0$ in $\mathcal{F}$.
		\item [{\rm(ii)}] For each compact subset $K$ of $X$, there exist a sequence of Borel sets $(D_{k})_{k=1}^\infty$ in $K$
		and a strictly increasing sequence $(n_k)\subseteq\Bbb{N}$ such that
		$\lim_{k\rightarrow \infty}\|\chi_{K\setminus D_{k}}\|_{\mathcal F}=0$ and
		$$\lim_{k\rightarrow \infty}\left\|\chi_{D_{k}}\,\left(\prod_{s=1}^{n_k}w\circ\alpha^{-s}\right)\right\|_{\mathcal F}=0.$$
	\end{enumerate}
\end{theorem}

\begin{proof}
	Assume that (i) holds. Let $K$ be a compact subset of $X$. Since $\alpha$ is an aperiodic function on $X$, there is a constant $N>0$ such that for each $n>N$,  $\alpha^{n}(K)\cap K=\varnothing ,$ which also gives $\alpha^{-n}(K) \cap K=\varnothing $ for all $n > N.$ By the assumption, there are a sequence $(n_k)\subseteq \mathbb{N}$ and a sequence $ (f_k)\subseteq \mathcal{F}$ with $N<n_k<n_{k+1}$ for all $k\in\mathbb{N}$, such that $\|f_k-\chi_K\|_{\mathcal F}<\dfrac{1}{4^{k}}$, $\|f_k+S_{\alpha,w}^{2n_k}f_k-\chi_K\|_{\mathcal F}<\dfrac{1}{4^k}$ and $C^{(n_k)}_{\alpha,w}f_k=f_k$.	
	For each $k=1,2,\ldots$, put $$D_k:=\{x\in K:\, \left|f_k(x)+S_{\alpha,w}^{2n_k}f_k(x)-1\right|\leq\dfrac{1}{2^k}\}.$$
	Then, we have
	\begin{align*}
		\dfrac{1}{2^{k}}\|\chi_{K\setminus D_k}\|_{\mathcal F}&\leq \left\|\chi_{K\setminus D_k}\cdot\left(f_k(x)+S_{\alpha,w}^{2n_k}f_k(x)-1\right)\right\|_{\mathcal F}\\
		&\leq \left\|f_k+S_{\alpha,w}^{2n_k}f_k-\chi_K\right\|_{\mathcal F}< \dfrac{1}{4^{k}}.
	\end{align*}
	
	This implies that $\|\chi_{K\setminus D_k}\|_{\mathcal F}<\dfrac{1}{2^{k}}$, and so $\lim_{k\rightarrow \infty}\|\chi_{K\setminus D_k}\|_{\mathcal F}=0$. On the other hand, since $C^{(n_k)}_{\alpha,w}f_k=f_k$, we have
	
	\begin{align*}
		\frac{2}{4^k}&\geq 2\left\|f_k-\chi_{K}\right\|_{\mathcal F}\\
		&=\left\|T^{n_k}_{\alpha,w}f_k+S^{n_k}_{\alpha,w}f_k-2\chi_K\right\|_{\mathcal F}\\
		&=\left\|\left(\prod_{s=0}^{n_k-1}w\circ\alpha^s\right)\cdot\left(f_k\circ\alpha^{n_k}\right)+\left(\prod_{s=1}^{n_k}w\circ \alpha^{-s}\right)^{-1}\,(f_k\circ\alpha^{-n_k})-2\chi_K\right\|_{\mathcal F}\\
		&\geq\left\|\chi_{\alpha^{-n_k}(K)}\cdot\left[\left(\prod_{s=0}^{n_k-1}w\circ\alpha^s\right)\cdot\left(f_k\circ\alpha^{n_k}\right)+\left(\prod_{s=1}^{n_k}w\circ \alpha^{-s}\right)^{-1}\,(f_k\circ\alpha^{-n_k})-2\chi_K\right]\right\|_{\mathcal F}\\
		&=\left\|\chi_{\alpha^{-n_k}(K)}\cdot\left[\left(\prod_{s=0}^{n_k-1}w\circ\alpha^s\right)\cdot\left(f_k\circ\alpha^{n_k}\right)+\left(\prod_{s=1}^{n_k}w\circ \alpha^{-s}\right)^{-1}\,(f_k\circ\alpha^{-n_k})\right]\right\|_{\mathcal F}\\
		&=\left\|(\chi_{\alpha^{-n_k}(K)}\circ\alpha^{-n_k})\cdot\left[\left(\prod_{s=0}^{n_k-1}w\circ\alpha^s\circ\alpha^{-n_k}\right)\cdot f_k+\left(\prod_{s=1}^{n_k}w\circ \alpha^{-s}\circ\alpha^{-n_k}\right)^{-1}\,(f_k\circ \alpha^{-2n_k})\right]\right\|_{\mathcal F}\\
		&=\left\|\chi_{K}\cdot\left[\left(\prod_{s=1}^{n_k}w\circ\alpha^{-s}\right)\cdot f_k+\left(\prod_{s=1}^{n_k}w\circ \alpha^{-(s+n_k)}\right)^{-1}\,(f_k\circ \alpha^{-2n_k})\right]\right\|_{\mathcal F}.
		\end{align*}
		
		This implies that 
		
		\begin{align*}
		\frac{2}{4^k}&\geq\left\|\chi_{K}\cdot\left[\left(\prod_{s=1}^{n_k}w\circ\alpha^{-s}\right)\cdot f_k+\left(\prod_{s=1}^{n_k}w\circ \alpha^{-(s+n_k)}\right)^{-1}\,(f_k\circ \alpha^{-2n_k})\right]\right\|_{\mathcal F}\\&=\left\|\chi_{K}\cdot\left(\prod_{s=1}^{n_k}w\circ\alpha^{-s}\right)\cdot \left[f_k+\left(\prod_{s=1}^{2n_k}w\circ \alpha^{-s}\right)^{-1}\,(f_k\circ \alpha^{-2n_k})\right]\right\|_{\mathcal F}\\
		&=\left\|\chi_{K}\cdot\left(\prod_{s=1}^{n_k}w\circ\alpha^{-s}\right)\cdot \left(f_k+S_{\alpha,w}^{2n_k}f_k\right)\right\|_{\mathcal F}\\
		&\geq\left\|\chi_{D_k}\cdot\left(\prod_{s=1}^{n_k}w\circ\alpha^{-s}\right)\cdot \left(f_k+S_{\alpha,w}^{2n_k}f_k\right)\right\|_{\mathcal F}\\
		&\geq \left(1-\frac{1}{2^{k}}\right)\left\|\chi_{D_k}\cdot\left(\prod_{s=1}^{n_k}w\circ\alpha^{-s}\right)\right\|_{\mathcal F}\\
	\end{align*}
	This completes the proof.
\end{proof}
\begin{corollary}\label{cor1}
	Let $\alpha$ be an aperiodic function on a topological space $X$, and $\mathcal F$ be a Banach function space on $X$ satisfying the condition $\Omega_{\alpha}$. Suppose that $X$ has a compact subset $K$ with $\|\chi_K\|_{\mathcal F}>0$, and let $ w,\frac{1}{w}$ be weight functions on $X$. If $\inf_{x\in X}w(x)>1$, 
then $\mathcal P((C_{\alpha,w}^{(n)})_{n\in \mathbb{N}_0})$ 
is not dense in $\mathcal{F}$, and so $(C^{(n)}_{\alpha,w})_{n\in \mathbb{N}_0}$ is not chaotic.
\end{corollary}
\begin{proof}
	For each $n\in\mathbb{N}$ and $f\in \mathcal{F}$,
	\begin{align*}
		\|S^{n}_{\alpha,w}f\|_{\mathcal F}&=\left\| \frac{1}{\prod_{s=1}^{n}w\circ\alpha^{-s}}\,(f\circ\alpha^{-n})\right\|_{\mathcal F}\\
		&=\left\| \frac{1}{\prod_{s=1}^{n}w\circ\alpha^{-s}\circ\alpha^n}\,f\right\|_{\mathcal F}\\
		&\leq \frac{1}{\left(\inf_{x\in X}w(x)\right)^{n}}\,\|f\|_{\mathcal F}.
	\end{align*}
	Hence, for each $f\in \mathcal{F}$, we have
	\begin{equation*}
		\lim_{n\rightarrow\infty}S^{n}_{\alpha,w}f=0
	\end{equation*}
	in $\mathcal{F}$.
	On the other hand,
	assume that $K$ is a compact subset of $X$ with $\|\chi_K\|_{\mathcal F}>0$, $(D_{k})_{k=1}^\infty$ is a sequence of Borel subsets of $K$ with
	$\lim_{k\rightarrow \infty}\|\chi_{K\setminus D_{k}}\|_{\mathcal F}=0$ and $(n_k)_{k=1}^\infty\subseteq\mathbb{N}$ is a strictly increasing sequence.
	Then
	\begin{equation*}
		\|\chi_{D_k}\|_{\mathcal F}\,\left(\inf_{x\in X}w(x)\right)^{n_k}\leq\left\|\chi_{D_{k}}\left(\prod_{s=1}^{n_k}w\circ\alpha^{-s}\right)\right\|_{\mathcal F},
	\end{equation*}
	and so 
	\begin{equation*}
		\lim_{k\rightarrow \infty}\left\|\chi_{D_{k}}\left(\prod_{s=1}^{n_k}w\circ\alpha^{-s}\right)\right\|_{\mathcal F}=\infty,
	\end{equation*}
	
	So, by Theorem \ref{main}, the proof is complete.
\end{proof}
\begin{theorem}\label{main2}
	Let $\alpha$ be an aperiodic function on a topological space $X$, and $\mathcal F$ be a Banach function space on $X$ satisfying the condition $\Omega_{\alpha}$. Let $ w,\frac{1}{w}$ be weight functions on $X$. Then, ${\rm (i)}\Rightarrow{\rm (ii)}$.
	\begin{enumerate}
		\item [{\rm(i)}] The set $\mathcal{P}((C_{\alpha,w}^{(n)})_{n\in \mathbb{N}_0})$ is dense in $\mathcal{F}$, and
		for each $f\in\mathcal{P}((C^{(n)}_{\alpha,w})_{n\in \mathbb{N}_0})$,  $\lim_{n\rightarrow\infty}T^{n}_{\alpha,w}f=0$ in $\mathcal{F}$.
		\item [{\rm(ii)}] For each compact set $K\subseteq X$, there are a sequence of Borel sets $(D_{k})_{k=1}^\infty$ in $K$
		and a strictly increasing sequence $(n_k)\subset\Bbb{N}$ such that
		$\lim_{k\rightarrow \infty}\|\chi_{K\setminus D_k}\|_{\mathcal F}=0$ and
		$$\lim_{k\rightarrow \infty}\left\|\chi_{D_{k}}\cdot\left(\prod_{s=0}^{n_k-1}w\circ\alpha^{s}\right)^{-1}\right\|_{\mathcal F}=0.$$
	\end{enumerate}
\end{theorem}

\begin{proof}
	The proof is similar to the proof of Theorem \ref{main}, but some details are different. Let $K$ be a compact subset of $X$.
	Since $\alpha$ is an aperiodic function on $X$, there is $ N > 0 $ such that for each $ n > N $,  $\alpha^{n}(K)\cap K=\varnothing $. By the assumption, there are a sequence $(n_k)\subset \mathbb{N}$ and a sequence $ (f_k)\subseteq \mathcal{F}$ with $N<n_1<n_2<\ldots$, such that $\|f_k-\chi_K\|_{\mathcal F}<\dfrac{1}{4^{k}}$, $\|f_k+T_{\alpha,w}^{2n_k}f_k-\chi_K\|_{\mathcal F}<\dfrac{1}{4^k}$ and $C^{(n_k)}_{\alpha,w}f_k=f_k$.
	For each $k=1,2,\ldots$, put $$D_k:=\{x\in K:\, \left|T^{2n_k}_{\alpha,w}f_k(x)+f_k(x)-1\right|\leq\dfrac{1}{2^k}\}.$$
	Then, we have
	\begin{align*}
		\dfrac{1}{2^{k}}\|\chi_{K\setminus D_k}\|_{\mathcal F}&\leq \|\chi_{K\setminus D_k}\cdot\left(T_{\alpha,w}^{2n_k}f_k(x)+f_k(x)-1\right)\|_{\mathcal F}\\
		&\leq \|T_{\alpha,w}^{2n_k}f_k(x)+f_k(x)-1\|_{\mathcal F}< \dfrac{1}{4^{k}}.
	\end{align*}
	
	This implies that $\lim_{k\rightarrow \infty}\|\chi_{K\setminus D_k}\|_{\mathcal F}=0$. Since $C^{(n_k)}_{\alpha,w}f_k=f_k$,
	\begin{align*}
		\frac{2}{4^k}&\geq2\left\|f_k-\chi_{K}\right\|_{\mathcal F}\\
		&=\left\|T^{n_k}_{\alpha,w}f_k+S^{n_k}_{\alpha,w}f_k-2\chi_K\right\|_{\mathcal F}\\
		&=\left\|\left(\prod_{s=0}^{n_k-1}w\circ\alpha^s\right)\cdot(f_k\circ\alpha^{n_k})+\left(\prod_{s=1}^{n_k}w\circ\alpha^{-s}\right)^{-1}\,(f_k\circ\alpha^{-n_k})-2\chi_K\right\|_{\mathcal F}\\
		&\geq\left\|\chi_{\alpha^{n_k}(K)}\cdot\left[\left(\prod_{s=0}^{n_k-1}w\circ\alpha^s\right)\cdot(f_k\circ\alpha^{n_k})+\left(\prod_{s=1}^{n_k}w\circ\alpha^{-s}\right)^{-1}\,(f_k\circ\alpha^{-n_k})-2\chi_K\right]\right\|_{\mathcal F}\\
		&=\left\|\chi_{\alpha^{n_k}(K)}\cdot\left[\left(\prod_{s=0}^{n_k-1}w\circ\alpha^s\right)\cdot(f_k\circ\alpha^{n_k})+\left(\prod_{s=1}^{n_k}w\circ\alpha^{-s}\right)^{-1}\,(f_k\circ\alpha^{-n_k})\right]\right\|_{\mathcal F}\\
		&=\left\|\chi_{\alpha^{n_k}(K)}\circ\alpha^{n_k}\cdot\left[\left(\prod_{s=0}^{n_k-1}w\circ\alpha^s\circ\alpha^{n_k}\right)\cdot(f_k\circ\alpha^{2n_k})+\left(\prod_{s=1}^{n_k}w\circ\alpha^{-s}\circ\alpha^{n_k}\right)^{-1}\,f_k\right]\right\|_{\mathcal F}\\
		&=\left\|\chi_{K}\cdot\left[\left(\prod_{s=0}^{n_k-1}w\circ\alpha^s\circ\alpha^{n_k}\right)\cdot(f_k\circ\alpha^{2n_k})+\left(\prod_{s=1}^{n_k}w\circ\alpha^{-s}\circ\alpha^{n_k}\right)^{-1}\,f_k\right]\right\|_{\mathcal F}\\
		&=\left\|\chi_{K}\cdot\left[\left(\prod_{s=0}^{n_k-1}w\circ\alpha^s\circ\alpha^{n_k}\right)\cdot(f_k\circ\alpha^{2n_k})+\left(\prod_{s=0}^{n_k-1}w\circ\alpha^{s}\right)^{-1}\,f_k\right]\right\|_{\mathcal F}.
		\end{align*}
		
		This implies that 
		
		\begin{align*}
		\frac{2}{4^k}&\geq \left\|\chi_{K}\cdot\left[\left(\prod_{s=0}^{n_k-1}w\circ\alpha^s\circ\alpha^{n_k}\right)\cdot(f_k\circ\alpha^{2n_k})+\left(\prod_{s=0}^{n_k-1}w\circ\alpha^{s}\right)^{-1}\,f_k\right]\right\|_{\mathcal F}\\
		&=\left\|\chi_{K}\left(\prod_{s=0}^{n_k-1}w\circ\alpha^{s}\right)^{-1}\cdot\left(T^{2n_k}_{\alpha,w}f_k+f_k\right)\right\|_{\mathcal F}\\
		&\geq\left\|\chi_{D_k}\left(\prod_{s=0}^{n_k-1}w\circ\alpha^{s}\right)^{-1}\cdot\left(T^{2n_k}_{\alpha,w}f_k+f_k\right)\right\|_{\mathcal F}\\
		&\geq\left(1-\frac{1}{2^{k}}\right)\,\left\|\chi_{D_k}\left(\prod_{s=0}^{n_k-1}w\circ\alpha^{s}\right)^{-1}\right\|_{\mathcal F}.
	\end{align*}
	This completes the proof.
\end{proof}
\begin{corollary}\label{cor2}
	Let $\alpha$ be an aperiodic function on a topological space $X$, and $\mathcal F$ be a Banach function space on $X$ satisfying the condition $\Omega_{\alpha}$. Suppose that $X$ has a compact subset $K$ with $\|\chi_K\|_{\mathcal F}>0$, and let $ w,\frac{1}{w}$ be weight functions on $X$. If $\sup_{x\in X}w(x)<1$, 
	then $\mathcal P((C_{\alpha,w}^{(n)})_{n\in \mathbb{N}_0})$ 
	is not dense in $\mathcal{F}$, and so $(C^{(n)}_{\alpha,w})_{n\in \mathbb{N}_0}$ is not chaotic.
\end{corollary}
\begin{proof}
	For each $n\in\mathbb{N}$ and $f\in \mathcal{F}$,
	\begin{align*}
	\|T^{n}_{\alpha,w}f\|_{\mathcal F}&=\left\| \left(\prod_{s=0}^{n-1}w\circ\alpha^{s}\right)\cdot(f\circ\alpha^{n})\right\|_{\mathcal F}\\
	&=\left\| \left(\prod_{s=0}^{n-1}w\circ\alpha^{s}\circ\alpha^{-n}\right)\cdot f\right\|_{\mathcal F}\\
	&=\left\| \left(\prod_{s=1}^{n}w\circ\alpha^{-s}\right)\cdot f\right\|_{\mathcal F}\\
	&\leq \left(\sup_{x\in X}w(x)\right)^{n}\,\|f\|_{\mathcal F}.
	\end{align*}
	So, by assumption, for each $f\in \mathcal{F}$ we have
	\begin{equation*}
	\lim_{n\rightarrow\infty}T^{n}_{\alpha,w}f=0
	\end{equation*}
	in $\mathcal{F}$.
	On the other hand,
	assume that $K$ is a compact subset of $X$ with $\|\chi_K\|_{\mathcal F}>0$, $(D_{k})_{k=1}^\infty$ is a sequence of Borel subsets of $K$ with
	$\lim_{k\rightarrow \infty}\|\chi_{K\setminus D_{k}}\|_{\mathcal F}=0$ and $(n_k)_{k=1}^\infty\subseteq\mathbb{N}$ is a strictly increasing sequence.
	Then
	\begin{equation*}
	\|\chi_{D_k}\|_{\mathcal F}\,\left(\sup_{x\in X}w(x)\right)^{-n_k}\leq\left\|\chi_{D_{k}}\cdot\left(\prod_{s=0}^{n_k-1}w\circ\alpha^{s}\right)^{-1}\right\|_{\mathcal F},
	\end{equation*}
	and so 
	\begin{equation*}
	\lim_{k\rightarrow \infty}\left\|\chi_{D_{k}}\cdot\left(\prod_{s=0}^{n_k-1}w\circ\alpha^{s}\right)^{-1}\right\|_{\mathcal F}=\infty,
	\end{equation*}
	
	and the proof is complete, thanks to Theorem \ref{main2}.
\end{proof}
From Corollaries \ref{cor1} and \ref{cor2} a necessary condition for the cosine operator function to be chaotic is given in below result.
\begin{corollary}\label{ttthm}
	Let $\alpha$ be an aperiodic function on a topological space $X$, and $\mathcal F$ be a Banach function space on $X$ satisfying the condition $\Omega_{\alpha}$. Suppose that $X$ has a compact subset $K$ with $\|\chi_K\|_{\mathcal F}>0$, and let $ w,\frac{1}{w}$ be weight functions on $X$. If $(C^{(n)}_{\alpha,w})_{n\in \mathbb{N}_0}$ is chaotic, then $\inf_{x\in X}w(x)\leq 1\leq \sup_{x\in X}w(x)$.
\end{corollary}
\begin{remark}
	 Note that in the proofs of Theorems \ref{main} and \ref{main2} and Corollaries \ref{cor1}, \ref{cor2} and \ref{ttthm} we have not used the assumption that compactly supported  bounded functions in $\mathcal F$ are dense in $\mathcal F$.
\end{remark}
Now, we give a sufficient condition for the cosine operator function to be topologically transitive.
\begin{theorem}\label{thm}
		Let $\alpha$ be a bijection on a topological space $X$ such that $\alpha$ and $\alpha^{-1}$ are Borel measurable, and $\mathcal F$ be a Banach function space on $X$ satisfying the condition $\Omega_{\alpha}$. Let $ w,\frac{1}{w}$ be weight functions on $X$. Then, {\rm (ii)} implies {\rm (i)}:
	\begin{enumerate}
		\item [{\rm(i)}] The sequence $(C_{\alpha,w}^{(n)})_{n\in \mathbb{N}_0} $ is  topologically transitive.
		\item [{\rm(ii)}]For each  compact subset $K\subseteq X$,  there are sequences $(D_k)$, $(E_k)$ and $(F_k)$ of Borel subsets of
		$K$, and  an increasing sequence $ (n_k) $ of positive integers  such that for each $k$, $D_k=E_k\cup F_k,$ $E_{k} \cap F_{k}= \varnothing,$ $\lim_{k\rightarrow \infty}\|\chi_{K\setminus D_{k}}\|_{\mathcal F}=0,$ 
		$$\lim_{k\rightarrow \infty}\left\|\chi_{D_{k}}\cdot\left(\prod_{s=1}^{n_k}w\circ \alpha^{-s}\right)\right\|_{\mathcal F}=\lim_{k\rightarrow \infty}\left\|\chi_{D_{k}}\cdot\left(\prod_{s=0}^{n_k-1}(w\circ \alpha^{s})\right)^{-1}\right\|_{\mathcal F}=0$$
		and 
		$$\lim_{k\rightarrow \infty}\left\|\chi_{E_{k}}\cdot\left(\prod_{s=1}^{2n_k}w\circ \alpha^{-s}\right)\right\|_{\mathcal F}=\lim_{k\rightarrow \infty}\left\|\chi_{F_{k}}\cdot\left(\prod_{s=0}^{2n_k-1}(w\circ \alpha^{s})\right)^{-1}\right\|_{\mathcal F}=0.$$
	\end{enumerate}
\end{theorem}
\begin{proof}
	
	Suppose that the condition (ii) holds, and $U$ and $V$ are non-empty open subsets of $\mathcal{F}$. Since $\mathcal{F}_{bc}$ is dense in $\mathcal F$, we can pick some compactly supported bounded functions $f\in U$ and $g\in V$. Let $K:=\text{supp}(f)\cup\text{supp}(g)$. Since $K$ is compact, there are sequences $(D_k)$, $(E_k)$, $(F_k)$ and $(n_k)$ satisfying the condition (ii) for the compact set $K$. Hence
	\begin{align*}
		\left\|T^{n_k}_{\alpha,w}(f\chi_{D_k})\right\|_{\mathcal F} &=\left\|\left(\prod_{s=0}^{n_k-1}w\circ \alpha^{s}\right)\,
		\cdot \left((f\chi_{D_k})\circ\alpha^{n_k}\right)\right\|_{\mathcal F}\\
		&\leq\sup_{x\in X}|f(x)|\, \left\|\left(\prod_{s=1}^{n_k}w\circ \alpha^{-s}\right)\,
		\cdot \chi_{D_k}\right\|_{\mathcal F} \rightarrow 0,
	\end{align*}
	as $k\rightarrow\infty$. This implies that $\lim_{k\longrightarrow \infty} T^{n_k}_{\alpha,w}(f\chi_{D_k})=0$ in $\mathcal F$. Similarly, we have
	\begin{align*}
		\lim_{k\longrightarrow \infty} S^{n_k}_{\alpha,w}(f\chi_{D_k})&=\lim_{k\longrightarrow \infty} T^{n_k}_{\alpha,w}(g\chi_{D_k})
		=\lim_{k\longrightarrow \infty} S^{n_k}_{\alpha,w}(g\chi_{D_k})\\
		&=\lim_{k\longrightarrow \infty} T^{2n_k}_{\alpha,w}(g\chi_{E_k})=\lim_{k\longrightarrow \infty} S^{2n_k}_{\alpha,w}(g\chi_{F_k})=0
	\end{align*}
	in $\mathcal F$. Moreover, using these equalities and the assumption that $\mathcal F$ is solid, we easily obtain that $\lim_{k\rightarrow\infty} T^{n_k}_{\alpha,w}(g \chi_{E_k})=\lim_{k\rightarrow \infty} S^{n_k}_{\alpha,w}(g \chi_{F_k})=0$  in $\mathcal F$. Therefore, setting
	$$v_k:=f\chi_{D_k}+2T^{n_k}_{\alpha,w}(g\chi_{E_k})+2S^{n_k}_{\alpha,w}(g\chi_{F_k}),$$
	for all $k$, we have
	$$\lim_{k\longrightarrow \infty}v_k=f\quad\text{ and }\lim_{k\longrightarrow \infty}C^{(n_k)}_{\alpha,w}v_k=g.$$
	
	This implies that there is an index $k$ with $C^{(n_k)}_{\alpha,w}v_k\in C^{(n_k)}_{\alpha,w}(U)\cap V$. So the condition (i) follows.
\end{proof}
Here, as an application of Theorem \ref{thm}, a sufficient condition for $(C_{\alpha,w}^{(n)})_{n\in \mathbb{N}_0}$ to be chaotic is stated.

\begin{corollary}\label{sufficient}
	Let $\alpha$ be a bijection on a topological space $X$ such that $\alpha$ and $\alpha^{-1}$ are Borel measurable, and $\mathcal F$ be a Banach function space on $X$ satisfying the condition $\Omega_{\alpha}$. Let $ w,\frac{1}{w}$ be weight functions on $X$. Then we have ${\rm(ii)} \Rightarrow {\rm(i)}.$
	\begin{enumerate}
		\item [{\rm(i)}] The sequence $(C_{\alpha,w}^{(n)})_{n\in \mathbb{N}_0}$ is chaotic on $\mathcal{F}$.
		\item [{\rm(ii)}] For each compact set $K\subseteq X$, there exist a sequence of Borel sets $(D_{k})_{k=1}^\infty$ in $K$
		and a strictly increasing sequence $(n_k)\subset\Bbb{N}$ such that
		$\lim_{k\rightarrow \infty}\|\chi_{K\setminus D_{k}}\|_{\mathcal F}=0$ and
		$$\lim_{k\rightarrow \infty}\sum_{l=1}^{\infty}\left\|\chi_{D_{k}}\cdot \Big(\prod_{s=1}^{ln_k}w\circ \alpha^{-s}\Big)\right\|_{\mathcal F}
		=\lim_{k\rightarrow \infty}\sum_{l=1}^{\infty}\left\|\chi_{D_{k}}\cdot\Big(\prod_{s=0}^{ln_k-1}w\circ \alpha^{s}\Big)^{-1}\right\|_{\mathcal F}=0,$$
	\end{enumerate}
 where the corresponding series are convergent for each $k$.
\end{corollary}
\begin{proof}
	By Theorem \ref{thm}, the condition (ii)  automatically implies that $(C_{\alpha,w}^{(n)})_{n\in \mathbb{N}_0}$ is topologically transitive.
	Hence, it would be sufficient to prove that the set $\mathcal{P}((C_{\alpha,w}^{(n)})_{n\in \mathbb{N}_0})$ is dense in $\mathcal{F}$.
	Let $U$ be a non-empty open subset of $\mathcal{F}$. Then there exists a compactly supported bounded function $f$ in $U$. Put $K:={\rm supp}(f)$. Then there exist a sequence $(D_{k})_{k=1}^\infty$ of Borel subsets of $K$, and a strictly increasing sequence $(n_k)\subseteq\Bbb{N}$ satisfying the condition (ii). Set
	$$v_k:=f\chi_{D_k}+\sum_{l=1}^{\infty}T^{ln_k}_{\alpha,w}(f\chi_{D_k})+\sum_{l=1}^{\infty}S^{ln_k}_{\alpha,w}(f\chi_{D_k}).$$
	
	Note that the series $\sum_{l=1}^{\infty}T^{ln_k}_{\alpha,w}(f\chi_{D_k})$ and $\sum_{l=1}^{\infty}S^{ln_k}_{\alpha,w}(f\chi_{D_k})$ are convergent in $\mathcal F$ because they are absolutely convergent by (ii). Then $\lim_{k\rightarrow\infty}v_k=f$ because
	\begin{eqnarray*}
		\|v_k-f\|_{\mathcal F}&\leq&\|f\|_\infty\left\|\chi_{K\setminus D_k}\right\|_{\mathcal F}+\sum_{l=1}^{\infty}\left\|T^{ln_k}_{\alpha,w}(f\chi_{D_k})\right\|_{\mathcal F}
		+\sum_{l=1}^{\infty}\left\|S^{ln_k}_{\alpha,w}(f\chi_{D_k})\right\|_{\mathcal F}\\
		&\leq&\|f\|_\infty \left\|\chi_{K\setminus D_k}\right\|_{\mathcal F}\\
		&&\hspace{-1cm}+\|f\|_\infty\sum_{l=1}^{\infty}\left\|\chi_{D_{k}}\cdot \Big(\prod_{s=1}^{ln_k}w\circ \alpha^{-s}\Big)\right\|_{\mathcal F}
		+\|f\|_\infty\sum_{l=1}^{\infty}\left\|\chi_{D_{k}}\cdot\Big(\prod_{s=0}^{ln_k-1}w\circ \alpha^{s}\Big)^{-1}\right\|_{\mathcal F}.
	\end{eqnarray*}
	Moreover, one can see that $C_{\alpha,w}^{(ln_k)}v_k=v_k$ for all $l\in\Bbb{N}$ by a simple calculation. Therefore,
	$$U\cap\mathcal{P}((C_{\alpha,w}^{(n)})_{n\in \mathbb{N}_0})\neq\varnothing$$
	which implies that $\mathcal{P}((C_{\alpha,w}^{(n)})_{n\in \mathbb{N}_0})$ is dense in $\mathcal{F}$.
\end{proof}
\section{The Adjoint of Cosine Functions}
Let $\mathcal{X}$ be a Banach space. The adjoint of an operator $T\in B(\mathcal{X})$ is defined by 
$$T^*:\mathcal{X}'\rightarrow\mathcal{X}',\quad T^*(\varphi)(x):=\varphi(T(x)),$$
where $\mathcal{X}'$ is the dual of $\mathcal{X}$, $\varphi\in\mathcal{X}'$ and $x\in\mathcal{X}$. 

For each $\varphi\in\mathcal{F}'$ and bounded function $g$ in $\mathcal{M}_0(X)$, we define $\varphi_g\in\mathcal{F}'$ by 
$$\varphi_g(f):=\varphi(gf),\quad(f\in\mathcal{F}).$$
For each $n\in\mathbb{N}$ and $\varphi\in \mathcal{F}'$ we have 
\begin{equation*}
T_{\alpha,w}^{n^*}(\varphi)=\varphi_{\prod_{s=0}^{n-1}(w\circ\alpha^s)}\circ U_{\alpha}^n\quad\text{and}\quad S_{\alpha,w}^{n^*}(\varphi)=\varphi_{(\prod_{s=1}^{n}(w\circ\alpha^{-s}))^{-1}}\circ U_{\alpha}^{-n},
\end{equation*}
where $U_{\alpha}:\mathcal{F}\rightarrow\mathcal{F}$ is defined by 
$U_{\alpha}(f):=f\circ\alpha$ for all $f\in\mathcal{F}$.

In this section, we give a sufficient condition for the sequence  $(C^{(n)^*}_{\alpha,w})_{n\in \mathbb{N}_0}$ to be chaotic on $\mathcal{F}'$.
\begin{theorem}\label{thmf}
	Suppose that there is a compact subset $K$ of $X$ such that $\alpha^n(K)\cap K=\varnothing$ for all $n\geq N$.
	Suppose that $\mathcal{P}((C^{(n)^*}_{\alpha,w})_{n\in \mathbb{N}_0})$ is dense in $\mathcal{F}'$, and $\lim_{n\rightarrow\infty} T_{\alpha,w}^{n^*}\varphi=0$ in $\mathcal{F}'$ for all $\varphi\in\mathcal{F}'$. Then, $\sup_{x\in X}w(x)>1$.
\end{theorem}
\begin{proof}
	Let $K$ be a compact subset of $X$ such that $\alpha^n(K)\cap K=\varnothing$ for all $n\geq N$. Since $\chi_K\in\mathcal{F}$, by Hahn-Banach Theorem, there exists some $\phi_K\in\mathcal{F}'$ such that $\phi_K(\chi_K)=\|\chi_K\|_{\mathcal{F}}$ and $\|\phi_K\|=1$. Define $\tilde{\phi}_K(f):=\phi_K(\chi_K f)$ for all $f\in \mathcal{F}$. Clearly, by properties of the function space $\mathcal{F}$, we have $\tilde{\phi}_K\in\mathcal{F}'$, and 
	$\|\tilde{\phi}_K\|=1$.

By the assumption, there are sequences $(\varphi_k)\subseteq \mathcal{F}'$ and $(n_k)\subseteq \mathbb{N}$ with $N<n_1<n_2<\ldots$ such that $C_{\alpha,w}^{(n_k)^*}(\varphi_k)=\varphi_k$,
\begin{equation*}
\|\varphi_k+T_{\alpha,w}^{*^{2n_k}}(\varphi_k)+\tilde{\phi}_K\|<\frac{1}{4^k},\quad\text{and}\quad\|\varphi_k-\tilde{\phi}_K\|<\frac{1}{4^k}
\end{equation*}
for all $k\in\mathbb{N}_0$. This implies that for each $k$ there exists some $0\neq f_k\in\mathcal{F}$ such that
$$\Big|\varphi_k(\chi_K f_k)+T_{\alpha,w}^{*^{2n_k}}(\varphi_k)(\chi_K f_k)\Big|>\Big(1-\frac{1}{4^k}\Big)\,\|f_k\|_{\mathcal{F}}.$$

For each $n\in\mathbb{Z}$ and bounded function $g\in \mathcal{M}_0(X)$,
$$\Big(\varphi\circ U^n_{\alpha}\Big)_{g}=\varphi_{g\circ\alpha^n}\circ U^n_{\alpha}.$$

Next, by some calculation we have 
\begin{align*}
\frac{2}{4^k}&\geq 2\|\varphi_k-\tilde{\phi}_K\|\\
&=\|T_{\alpha,w}^{*^{n_k}}(\varphi_k)+S_{\alpha,w}^{*^{n_k}}(\varphi_k)-2\tilde{\phi}_K\|\\
&=\|(\varphi_k)_{\prod_{s=0}^{n_k-1}(w\circ\alpha^s)}\circ U_{\alpha}^{n_k}+(\varphi_k)_{(\prod_{s=1}^{n_k}(w\circ\alpha^{-s}))^{-1}}\circ U_{\alpha}^{-n_k}-2\tilde{\phi}_K\|\\
&=\|(\varphi_k)_{\prod_{s=0}^{n_k-1}(w\circ\alpha^s)}\circ U_{\alpha}^{2n_k}+(\varphi_k)_{(\prod_{s=1}^{n_k}(w\circ\alpha^{-s}))^{-1}}-2\tilde{\phi}_K\circ U_{\alpha}^{n_k}\|\\
&=\|\Big((\varphi_k)_{\prod_{s=0}^{2n_k-1}(w\circ\alpha^s)}\circ U_{\alpha}^{2n_k}\Big)_{\prod_{s=1}^{n_k}(w\circ\alpha^{-s})^{-1}}+(\varphi_k)_{(\prod_{s=1}^{n_k}(w\circ\alpha^{-s}))^{-1}}-2\tilde{\phi}_K\circ U_{\alpha}^{n_k}\|\\
&=\|\Big(T_{\alpha,w}^{*^{2n_k}}(\varphi_k)+\varphi_k\Big)_{(\prod_{s=1}^{n_k}(w\circ\alpha^{-s}))^{-1}}-2\tilde{\phi}_K\circ U_{\alpha}^{n_k}\|\\
&\geq\|\Big(\Big(T_{\alpha,w}^{*^{2n_k}}(\varphi_k)+\varphi_k\Big)_{(\prod_{s=1}^{n_k}(w\circ\alpha^{-s}))^{-1}}-2\tilde{\phi}_K\circ U_{\alpha}^{n_k}\Big)_{\chi_K}\|\\
&=\|\Big(\Big(T_{\alpha,w}^{*^{2n_k}}(\varphi_k)+\varphi_k\Big)_{\chi_K\,(\prod_{s=1}^{n_k}(w\circ\alpha^{-s}))^{-1}}\|\\
&=\Big\|\Big(T_{\alpha,w}^{*^{2n_k}}(\varphi_k)+\varphi_k\Big)_{\chi_K W_k^{-1}}\Big\|,
\end{align*}
where $W_k:=\prod_{s=0}^{n_k-1}(w\circ \alpha^{s})$. Now, setting 
$\widetilde{f_k}:=W_k f_k$, we have 
\begin{align*}
\frac{2}{4^k}\,\|\widetilde{f_k}\|_{\mathcal{F}}&\geq \Big\|\Big(T_{\alpha,w}^{*^{2n_k}}(\varphi_k)+\varphi_k\Big)_{\chi_K W_k^{-1}}\Big\|\,\|\widetilde{f_k}\|_{\mathcal{F}}\\
&\geq \Big|\Big(T_{\alpha,w}^{*^{2n_k}}(\varphi_k)+\varphi_k\Big)_{\chi_K W_k^{-1}}(\widetilde{f_k})\Big|\\
&=\Big|(T_{\alpha,w}^{*^{2n_k}}(\varphi_k)+\varphi_k)(\chi_K f_k)\Big|\\
&>\Big(1-\frac{1}{4^k}\Big)\,\|f_k\|_{\mathcal{F}}\\
&=\Big(1-\frac{1}{4^k}\Big)\,\|W_k^{-1}\widetilde{f_k}\|_{\mathcal{F}}\\
&\geq\Big(1-\frac{1}{4^k}\Big)\,\,\|\widetilde{f_k}\|_{\mathcal{F}}\,\inf_{x\in X}W_k^{-1}(x)\\
&\geq\Big(1-\frac{1}{4^k}\Big)\,\,\|\widetilde{f_k}\|_{\mathcal{F}}\,(\sup_{x\in X}w(x))^{-n_k}.
\end{align*}
Dividing by $\|\widetilde{f_k}\|_{\mathcal{F}}$ on the both sides of this inequality, we get
$$\frac{2}{4^k}\geq \Big(1-\frac{1}{4^k}\Big)\,\,(\sup_{x\in X}w(x))^{-n_k},$$
which gives $w(x)>1$ for all $x\in X$.
\end{proof}
Similarly, one can prove the following result.
\begin{theorem}
	Suppose that there is a compact subset $K$ of $X$ such that $\alpha^n(K)\cap K=\varnothing$ for all $n\geq N$.
	Suppose that $\mathcal{P}((C^{(n)^*}_{\alpha,w})_{n\in \mathbb{N}_0})$ is dense in $\mathcal{F}'$, and $\lim_{n\rightarrow\infty} S_{\alpha,w}^{n^*}\varphi=0$ in $\mathcal{F}'$ for all $\varphi\in\mathcal{F}'$. Then, $\inf_{x\in X}w(x)<1$.
\end{theorem}
\begin{corollary}
	Suppose that there is a compact subset $K$ of $X$ such that $\alpha^n(K)\cap K=\varnothing$ for all $n\geq N$.
	If $(C^{(n)^*}_{\alpha,w})_{n\in \mathbb{N}_0}$ is chaotic, then $\inf_{x\in X}w(x)<1<\sup_{x\in X}w(x)$.
\end{corollary}
\begin{proof}
	For each $\varphi\in\mathcal{F}'$ we have 
	$$\|T_{\alpha,w}^{*^n}(\varphi)\|=\|\varphi_{\prod_{s=0}^{n-1}(w\circ\alpha^s)}\|\leq \|\varphi\|\,\sup_{x\in X}\prod_{s=0}^{n-1}(w\circ\alpha^s)(x)\leq \|\varphi\|\, (\sup_{x\in X}w(x))^n.$$
	Similarly, 
	$$\|S_{\alpha,w}^{*^n}(\varphi)\|\leq \|\varphi\|\, (\inf_{x\in X}w(x))^{-n}.$$
	So, if $\sup_{x\in X}w(x)<1$, then $\lim_{n\rightarrow\infty}T_{\alpha,w}^{*^n}(\varphi)=0$ for all $\varphi\in\mathcal{F}'$, and this contradicts Theorem \ref{thmf}.  Similarly, if $\inf_{x\in X}w(x)>1$, then $\lim_{n\rightarrow\infty}S_{\alpha,w}^{*^n}(\varphi)=0$ for all $\varphi\in\mathcal{F}'$, and we get a contradiction.
\end{proof}

\vspace{.1in}

\end{document}